\documentclass[a4paper,abstract=yes]{scrartcl}

\usepackage{amsthm}
\usepackage{hyperref}
\usepackage{microtype}

\newtheorem{thm}{Theorem}
\newtheorem{lem}[thm]{Lemma}
\newtheorem{conj}[thm]{Conjecture}

\begin{document}

\title{An elementary conjecture which implies the Goldbach conjecture}
\author{Richard Williamson\thanks{Email: \href{mailto:richard@rwilliamson-mathematics.info}{richard@rwilliamson-mathematics.info}}}
\date{\today}

\maketitle

\begin{abstract} Let $p_{1}$, \ldots, $p_{k}$ be the first $k$ odd primes in succession. Let $n$ be an even integer such that $n > p_{k}$. We conjecture that if none of $n - p_{1}$, \ldots, $n - p_{k}$ are prime, then at least one of them has a prime factor which is greater than or equal to $p_{k}$. In this brief note, we observe that Goldbach's conjecture follows from this conjecture. \end{abstract}

Questions concerning prime factors of a collection of composite integers in a certain range can be hard to answer. A notable example, discussed for instance in \cite{ErdosPomeranceAnAnalogueOfGrimmsProblemOfFindingDistinctPrimeFactorsOfConsecutiveIntegers}, is Grimm's conjecture that given $k$ consecutive composite integers $n_{1}$, \ldots, $n_{k}$, one can find $k$ distinct primes $q_{1}$, \ldots, $q_{k}$ such that, for each $1 \leq i \leq k$, $n_{i}$ is divisible by $q_{i}$.

The purpose of this note is to bring attention to the following.

\begin{conj} \label{Conjecture} Let $p_{1}$, \ldots, $p_{k}$ denote the first $k$ odd primes in succession. Let $n$ be an even integer such that $n > p_{k}$. Suppose that for all $1 \leq i \leq k$, $n - p_{i}$ is not prime. Then there is an integer $i$ with $1 \leq i \leq k$ for which $n - p_{i}$ has a prime factor $p$ such that $p \geq p_{k}$. \end{conj}

Computer calculations support that Conjecture \ref{Conjecture} holds. Indeed, the first $i$ for which it holds is, in all cases which have been considered, much smaller than $k$, except of course when $k$ itself is very small.

It seems that the case in which $p$ is equal to $p_{k}$ can occur only in very special circumstances. It occurs when $k=1$ for any $n$ of the form $3^{r} + 3$ for some $r > 1$. It also occurs when $k=2$ and $n=30$. No other examples are known to the author.

Characterising exactly when $p$ can be equal to $p_{k}$ would be interesting. One reason for this will be raised later. Another is that if one wishes to prove Conjecture \ref{Conjecture} inductively, at least if one uses only elementary techniques, a precise characterisation of when $p$ can be equal to $p_{k}$ seems crucial.

Our reason for bringing Conjecture \ref{Conjecture} to light is that one can deduce from it, by a straightforward argument, that Goldbach's conjecture holds. We give the proof below. The author is not aware of an attempt having previously been made to approach the Goldbach conjecture in this way. If that is so, it would be remarkable that such a simple argument has hitherto been overlooked.

The deduction of Goldbach's conjecture from Conjecture \ref{Conjecture} relies on the following.

\begin{lem}[Assuming Conjecture \ref{Conjecture}] \label{Lemma} Let $n$ and $p_{1}$, $\ldots$, $p_{k}$ be as in the statement of Conjecture \ref{Conjecture}. Suppose that for all $1 \leq i \leq k$, $n - p_{i}$ is not prime. Then there is a prime $p$ such that $p_{k} < p < n$. \end{lem}

\begin{proof} By Conjecture \ref{Conjecture}, there is an integer $i$ with $1 \leq i \leq k$ for which $n - p_{i}$ has a prime factor $q$ such that $q \geq p_{k}$. In particular, we have that $q < n - p_{i} < n$. If $q > p_{k}$, we thus can take $p$ to be $q$.

Suppose that $q = p_{k}$. Since $n - p_{i}$ is not prime, we have that $n - p_{i} = mp_{k}$ for some odd integer $m \geq 3$. Thus we have that $3p_{k} \leq mp_{k} = n - p_{i} < n$. Now, by Bertrand's postulate, we have that $p_{k+1} < 2p_{k}$. Together, these inequalities demonstrate that $p_{k+1} < n$. Since $p_{k} < p_{k+1}$, we thus can take $p$ to be $p_{k+1}$.

\end{proof}

From Lemma \ref{Lemma} we deduce the following, which is Goldbach's conjecture.

\begin{thm}[Assuming Conjecture \ref{Conjecture}] \label{Theorem} Let $n \geq 6$ be an even integer. Then there is a prime $p < n$ such that $n-p$ is prime. \end{thm}

\begin{proof} Since there are only finitely many primes $q$ such that $q < n$, this follows immediately, by induction, from Lemma \ref{Lemma}.  \end{proof}

It seems likely that the use of Bertrand's postulate in the proof Lemma \ref{Lemma} is not fundamental. Given a sharper characterisation of when the $p$ of Lemma \ref{Lemma} can be equal to $p_{k}$, as discussed above, it should be possible to avoid it.

How difficult is Conjecture \ref{Conjecture}? This remains to be seen! Attempting to prove it inductively, using only elementary techniques, one can come close, but a gap manifests itself in the various approaches the author has taken. It is unclear whether this gap is fundamental, or whether it can be bridged.

It may well be that techniques from sieve theory, or other estimative methods, can make short work of Conjecture \ref{Conjecture}. We hope that experts on such techniques will, upon reading this note, give such an attempt a try.

The author is extremely grateful to several people for checking the above arguments, and for providing very useful feedback.

\bibliography{ref}
\bibliographystyle{siam}

\end{document}